\documentclass{article}
\usepackage{amsfonts}
\usepackage{amsmath}

\newtheorem{theorem}{Theorem}

\newtheorem{definition}[theorem]{Definition}

\newtheorem{lemma}[theorem]{Lemma}

\newtheorem{proposition}[theorem]{Proposition}

\newenvironment{proof}[1][Proof]{\noindent\textbf{#1.} }{\ \rule{0.5em}{0.5em}}
\input{tcilatex}
\begin{document}

\title{Equicontinuous factors of one dimensional cellular automata}
\author{Rezki CHEMLAL. \and Laboratoire de Math\'{e}matiques Appliqu\'{e}es,
\and Universit\'{e} Abderahmane Mira Bejaia.06000 Bejaia Algeria.}
\maketitle

\begin{abstract}
We are interested in topological and ergodic properties of one dimensional
cellular automata. We show that an ergodic cellular automaton cannot have
irrational eigenvalues. We show that any cellular automaton with an
equicontinuous factor has also as a factor an equicontinuous cellular
automaton. We show also that a cellular automaton with almost equicontinuous
points according to Gilman's classification has an equicontinuous measurable
factor which is a cellular automaton.\newline
\end{abstract}

\section*{Introduction}

Cellular automata dates back to John Von Neumann in the late 1940s. The
first models were biologically motivated. During the past years, they have
drawn lot of interest as models of complex systems \cite{Chop03},\cite{ABM03}%
,\cite{FarDen08},\cite{Nowak}. In the meantime a dual work on theoretical
properties of cellular automata was also attracting many researchers. These
properties can be studied using many tools and perspectives (language
recognition, dynamical systems, decidability \ldots etc.).

A cellular automaton is made of an infinite lattice of finite identical
automata. The lattice is usually $\mathbb{Z}^{n}$ with $n$ called the
dimension of the cellular automaton. The set of possible states of an
automaton is called the alphabet and each element of the alphabet is
referred to as a letter. A configuration is a snapshot of the state of all
automata in the lattice. The global state of the CA, specified by the values
of all the variables at a given time, evolves according to a global
transition map $F$ based on a local rule $f$ which acts on the value of each
single cell in synchronous discrete time steps.

Dynamical behavior of cellular automata is studied mainly in the context of
discrete dynamical systems by equipping the space of configurations with the
product topology which make it homeomorphic to the Cantor space.

We want to characterize equicontinuous factors of cellular automata both
measurable and topological. We establish a topological result : A cellular
automaton with an equicontinuous factor has also an equicontinuous cellular
automaton as a factor.

Two ergodic results are shown. We give a constructive proof of the existence
of a measurable equicontinuous factor for a cellular automata with
measurable equicontinuous points according to Gilman's classification.

Finally we show that an ergodic cellular automaton cannot be conjugated with
an irrational rotation.

The paper is divided into two parts, part one deals with basic notation and
definitions. The second part is for new results.

\section{Basic notions}

In this section, we recall standard definitions about CA as dynamical
systems. We begin by introducing some general notation we will use
throughout the rest of the paper.

\subsection{Cellular automata as dynamical systems}

Let $A$ be a finite set; a word is a sequence of elements of $A$. The length
of a finite word $u=u_{0}...u_{n-1}\in A^{n}$ is $\left\vert u\right\vert
=n. $We denote by $A^{\mathbb{Z}}$ the set of bi-infinite sequences over $A$%
. A point $x$ $\in $ $A^{\mathbb{Z}}$ is called a configuration. For two
integers $i,j$ with $i<j$ we denote by $x\left( i,j\right) $ the word $%
x_{i}...x_{j}.$

For any word $u$ we define the cylinder $\left[ u\right] _{l}=\left\{ x\in
A^{\mathbb{Z}}:x\left( l,l+\left\vert u\right\vert \right) =u\right\} $
where the word $u$ is at the position $l.$ The cylinder $\left[ u\right]
_{0} $ is simply noted $\left[ u\right] $. The cylinders are clopen (closed
open) sets.

Endowed with the distance $d\left( x,y\right) =2^{-n}$ with $n=min\left\{
i\geq 0:x_{i}\neq y_{i}\,\right. $\newline
$\left. \mathrm{or}\,x_{-i}\neq y_{-i}\right\} $, the set $A^{\mathbb{Z}}$
is a topological compact separated space.

The shift map $\sigma :$ $A^{\mathbb{Z}}\rightarrow $ $A^{\mathbb{Z}}$ is
defined as $\sigma \left( x\right) _{i}=x_{i+1},$ for any $x\in A^{\mathbb{Z}%
}$ and $i\in \mathbb{Z}$. The shift map is a continuous and bijective
function on $A^{\mathbb{Z}}.$ The dynamical system $\left( A^{\mathbb{Z}%
},\sigma \right) $ is commonly called \emph{full shift.}

A cellular automaton is a continuous map $F:A^{\mathbb{Z}}\rightarrow A^{%
\mathbb{Z}}$ commuting with the shift. By the Curtis-Hedlund Lyndon theorem 
\cite{Hed69} for every cellular automaton $F$ there exist an integer $r$ and
a block map $f$ from $A^{2r+1}$ to $A$ such that $F\left( x\right)
_{i}=f\left( x_{i-r},...,x_{i},...x_{i+r}\right) .$ The integer $r$ is
called the radius of the cellular automaton.

A point $x$ is said periodic if there exists $p>0$ with $F^{p}\left(
x\right) =x.$ The least $p$ with this property is called the period of $x.$
A\ point $x$ is eventually periodic if $F^{m}\left( x\right) $ is periodic
for some $m\geq 0.$

By commutation with the shift, every shift-periodic point is $F-$eventually
periodic and the set of eventually periodic points is dense.

Endowed with the sigma-algebra on $A^{\mathbb{Z}}$ generated by all cylinder
sets and $\nu $ the uniform measure which gives the same probability to
every letter of the alphabet, $\left( A^{\mathbb{Z}},\mathbb{B},F,\nu
\right) $ is a measurable space. The uniform measure is invariant if and
only if the cellular automaton is surjective \cite{Hed69}.In the following $%
\left( A^{\mathbb{Z}},\mathbb{B},F,\nu \right) $ will denote a surjective
cellular automaton equipped with the uniform measure.

\subsection{Equicontinuous and almost equicontinuous points of cellular
automata}

\subsubsection{K\r{u}rka's classification}

K\r{u}rka \cite{Kur03} introduced a topological classification based on the
equicontinuity, sensitiveness and expansiveness properties. The existence of
an equicontinuous point is equivalent to the existence of a blocking word
i.e. a configuration that stop the propagation of the perturbations on the
one dimensional lattice.

\begin{definition}
Let $F$ be a cellular automaton.\newline
1. A point $x$ is an equicontinuous point if : 
\begin{equation*}
\forall \epsilon >0,\exists \delta >0,\forall y:d\left( x,y\right) <\delta
,\forall n\geq 0,d\left( F^{n}\left( y\right) ,F^{n}\left( x\right) \right)
<\epsilon .
\end{equation*}%
\newline
2. We say that $F$ is equicontinuous if every point $x\in A^{\mathbb{Z}}$ is
an equicontinuous point.\newline
3.We say that $F$ is sensitive if for all $x\in A^{\mathbb{Z}}$ we have : 
\begin{equation*}
\exists \epsilon >0,\forall \delta >0,\exists y:d\left( x,y\right) <\delta
,\exists n\geq 0\text{,}d\left( F^{n}\left( y\right) ,F^{n}\left( x\right)
\right) \geq \epsilon .
\end{equation*}
\end{definition}

\begin{definition}
Let $F$ be a cellular automaton. A word $w$ with $\left\vert w\right\vert
\geq s$ is an $s$-blocking word for $F$ if there exists $p\in \left[
0,\left\vert w\right\vert -s\right] $ such that for any $x,y\in \left[ w%
\right] $ we have $F^{n}\left( x\right) \left( p,p+s\right) =F^{n}\left(
y\right) \left( p,p+s\right) $ for all $n\geq 0.$
\end{definition}

\begin{proposition}
Let $F$ be a cellular automaton\ with radius $r>0.$ The following conditions
are equivalent.\newline
1. $F$ is not sensitive. \newline
2. $F$ has an $r-$blocking word.\newline
3. $F$ has some equicontinuous point.
\end{proposition}

\subsubsection{Gilman's classification}

Based on the Wolfram's work \cite{Wol84}, Gilman \cite{Gil87}\cite{Gil88}
introduced a classification using Bernoulli measures which are not
necessarily invariant. Cellular automata can then be divided into three
classes : CA with equicontinuous points, CA with almost equicontinuous
points but without equicontinuous points and almost expansive CA.

In \cite{Tis08} Tisseur extends the Gilman's classification to any shift
ergodic measure and gives an example of a cellular automaton with an
invariant measure which have almost equicontinuous points but without
equicontinuous points.

\begin{definition}
Let $F$ be a cellular automaton and $\left[ i_{1},i_{2}\right] $ a finite
interval of $\mathbb{Z}$. For $x\in A^{\mathbb{Z}}$. We define $B_{\left[
i_{1},i_{2}\right] }\left( x\right) $ by : 
\begin{equation*}
B_{\left[ i_{1},i_{2}\right] }\left( x\right) =\left\{ y\in A^{\mathbb{Z}%
},\forall j:F^{j}\left( x\right) \left( i_{1},i_{2}\right) =F^{j}\left(
y\right) \left( i_{1},i_{2}\right) \right\} .
\end{equation*}
\end{definition}

For any interval $\left[ i_{1},i_{2}\right] $ the relation $\mathfrak{R}$
defined by $x\mathfrak{R}y$ if and only if $\forall j:F^{j}\left( x\right)
\left( i_{1},i_{2}\right) =F^{j}\left( y\right) \left( i_{1},i_{2}\right) $
is an equivalence relation and the sets $B_{\left[ i_{1},i_{2}\right]
}\left( x\right) $ are the equivalence classes.

\begin{definition}
Let $\left( F,\mu \right) $ a cellular automaton equipped with a shift
ergodic measure $\mu ,$ a point $x$ is $\mu -$equicontinuous if for any $m>0$
we have :%
\begin{equation*}
\underset{n\rightarrow \infty }{\lim }\frac{\mu \left( \left[ x\left(
-n,n\right) \right] \cap B_{\left[ -m,m\right] }\left( x\right) \right) }{%
\mu \left( \left[ x\left( -n,n\right) \right] \right) }=1.
\end{equation*}

We say that $F$ is $\mu -$almost expansive if there exist $m>0$ such that
for all $x\in A^{\mathbb{Z}}:\mu \left( B_{\left[ -m,m\right] }\left(
x\right) \right) =0.$
\end{definition}

\begin{definition}
Let $\left( F,\mu \right) $ denote a cellular automaton equipped with a
shift ergodic measure $\mu .$ Define classes of cellular automata as follows
:\newline
1- $\left( F,\mu \right) \in \mathcal{A}$ if $F$ is equicontinuous at some $%
x\in A^{\mathbb{Z}}.$\newline
2- $\left( F,\mu \right) \in \mathcal{B}$ if $F$ is $\mu -$almost
equicontinuous at some $x\in A^{\mathbb{Z}}$ but $F\notin \mathcal{A}$.%
\newline
3- $\left( F,\mu \right) \in \mathcal{C}$ if $F$ is $\mu -$almost expansive.
\end{definition}

\subsection{Measurable Dynamics}

A cellular automaton $\left( A^{\mathbb{Z}},\mathbb{B},F,\mu \right) $ is
ergodic if every invariant subset of $A^{\mathbb{Z}}$ is either of measure 0
or of measure 1. Equivalently, if for any measurable $U,V\subset A^{\mathbb{Z%
}},$ there exists some $n\in \mathbb{N}$ such that $\mu \left( U\cap
F^{-n}\left( V\right) \right) >0.$ It is said weakly mixing if $F\times F$
is ergodic.

A cellular automaton $\left( A^{\mathbb{Z}},\mathbb{B},F,\mu \right) $ is
mixing if , for any measurable $U,V\subset A^{\mathbb{Z}}$ we have :%
\begin{equation*}
\underset{n\rightarrow \infty }{\lim }\mu \left( U\cap F^{-n}\left( V\right)
\right) =\mu \left( U\right) \text{ }.\text{ }\mu \left( V\right)
\end{equation*}

A cellular automaton $(B^{\mathbb{Z}},G)$ is a topological factor of $(A^{%
\mathbb{Z}},F)$, if there exists a surjective continuous map $\pi $ from $A^{%
\mathbb{Z}}$ to $B^{\mathbb{Z}}$ such that $\pi \circ f=g\circ \pi $. We can
define in a similar way measurable factor if $\pi $ is a measurable map.

We denote by $L_{\mu }^{2}$ the set of measurable functions $g:A^{\mathbb{Z}%
}\rightarrow \mathbb{C}$ for which $\left\Vert f\right\Vert _{2}=\left(
\int_{A^{\mathbb{Z}}}\left\vert g\right\vert ^{2}d\mu \right) ^{\frac{1}{2}}$
is finite.

Let $\left( A^{\mathbb{Z}},\mathbb{B},F,\mu \right) $ be a cellular
automaton where $\mu $ is an invariant measure. We say that the function $%
g\in L_{\mu }^{2}$ is a measurable eigenfunction associated to the
measurable eigenvalue $\lambda \in \mathbb{C}$ if $g\circ F=\lambda $ $.$ $%
g\,ae$.

By definition any eigenvalue must be an element of the unit circle. The set $%
\mathbb{S}_{F}$ of all eigenvalues of $F$ which form a multiplicative sub
group of the complex roots of the unity is called its spectrum.

As any eigenvalue can be written in the form $\exp \left( 2i\pi \alpha
\right) $; we will say that an eigenvalue is rational if $\alpha \in \mathbb{%
Q}$ and irrational otherwise.

We say that $F$ has a discrete spectrum if $L_{\mu }^{2}$ is spanned by the
set of eigenfunctions of $F.$

A cellular automaton is ergodic iff any eigenfunction is of constant module
and weakly mixing iff it admits 1 as unique eigenvalue and that all
eigenfunctions are constant. For more details about these classical results
you can see for example \cite{Wal82}.

There is a relation between the spectrum of the shift and that of the
cellular automaton.

\begin{proposition}
\textbf{Pivato} \cite{Piv12} Let $\left( A^{\mathbb{Z}},F\right) $ be a
cellular automaton and $\mu $ a $\sigma -$ergodic measure we have :\newline
1- $\mathbb{S}_{\sigma }\subset \mathbb{S}_{F}.$\newline
2- If $\left( A^{\mathbb{Z}},\sigma ,\mu \right) $ has discrete spectrum,
then so does $\left( A^{\mathbb{Z}},F,\mu \right) .$\newline
3- If $\mu $ is $F-$ergodic and $\left( A^{\mathbb{Z}},\sigma ,\mu \right) $
is weakly mixing then so is $\left( A^{\mathbb{Z}},F,\mu \right) .$
\end{proposition}

\section{Statement of results}

This section is for new results. The first proposition states that if a
cellular automaton has a topological equicontinuous factor then it has also
a factor which is an equicontinuous cellular automaton.

\begin{proposition}
Let $F$ be a cellular automaton, if $F$ has an equicontinuous factor then $F$
has also as a factor an equicontinuous cellular automaton.
\end{proposition}

\begin{proof}
Let $G$ be an equicontinuous factor of $F$. There exists then a surjective
continuous map $\pi $ such that :%
\begin{equation*}
\pi \circ F=G\circ \pi
\end{equation*}

As $G$ is an equicontinuous map on a zero-dimensional space then it is
almost periodic. Let us denote the preperiod by $p_{0}$ and the period by $%
p. $ The set of possible values taken by $G$ is denoted by :%
\begin{equation*}
\mathbb{P}=\left\{ y_{k}:p_{0}\leq k\leq p_{0}+p-1\right\} .
\end{equation*}%
\newline
By uniform continuity there is an integer $N$ such that the function $G$
gives for each cylinder of length $N$ a value from $\mathbb{P}$. \newline
The set $G^{-1}\left( \left\{ y_{k}\right\} \right) $ can then be written in
the form $G^{-1}\left( \left\{ y_{k}\right\} \right) =\cup _{j=1}^{n_{k}}%
\left[ w_{j}^{k}\right] $ where $w_{j}^{k}$ is a word of length $N.$

Let us denote by $Y_{k}=\left\{ w_{j}^{k},1\leq j\leq n_{k}\right\}
,p_{0}\leq k\leq p_{0}+p-1.$ Define the alphabet : $B=\left\{
0,...,p\right\} ^{\mathbb{Z}}$ and the function $\pi ^{\prime }$ from $A^{%
\mathbb{Z}}$ to $B^{\mathbb{Z}}$ by :%
\begin{equation*}
\forall x\in A^{\mathbb{Z}}:\left\{ 
\begin{array}{l}
\pi ^{\prime }\left( x_{i}\right) =i-p_{0}:x_{i,i+N}\in Y_{k},p_{0}\leq
k\leq p_{0}+p-1 \\ 
p:otherwise.%
\end{array}%
\right.
\end{equation*}%
Notice that the function $\pi ^{\prime }$ is a surjection from $A^{\mathbb{Z}%
}$ to $B^{\mathbb{Z}}$.

Let us define the periodic CA $\left( \left\{ 0,...,p\right\} ^{\mathbb{Z}%
},C\right) $ by 
\begin{equation*}
\forall x\in \left( \mathbb{Z}/\left( p+1\right) \mathbb{Z}\right) ^{\mathbb{%
Z}}:C\left( x\right) =\left\{ 
\begin{array}{l}
\left( x_{i}+1\right) \,mod\,p\text{ if }x\neq p \\ 
p\text{ if }x=p%
\end{array}%
\right.
\end{equation*}%
Computing $\pi ^{\prime }\circ G$ and $C\circ \pi ^{\prime }$ we obtain :%
\begin{equation*}
\left( \pi ^{\prime }\circ G\right) \left( x\right) =\left\{ 
\begin{array}{l}
\left( x_{i}+1\right) \,mod\,p:x_{i,i+N}\in Y_{k},p_{0}\leq k\leq p_{0}+p-1
\\ 
p:otherwise.%
\end{array}%
\right.
\end{equation*}%
\begin{equation*}
\left( C\circ \pi ^{\prime }\right) \left( x\right) =\left\{ 
\begin{array}{l}
\left( x_{i}+1\right) \,mod\,p:x_{i,i+N}\in Y_{k},p_{0}\leq k\leq p_{0}+p-1
\\ 
p:otherwise.%
\end{array}%
\right.
\end{equation*}%
This shows that $C$ is an equicontinuous factor of $G.$ Hence $C$ is an
equicontinuous factor of $F.$
\end{proof}

The next result is about the existence of measurable equicontinuous factor
which is a cellular automaton for cellular automata with almost
equicontinuous points.

We start by the proof of a lemma. The main idea of the proof is not new \cite%
{Gil88}, but necessary to understand the proof of the following proposition.

\begin{lemma}
\label{Lemma} Let $\left( F,\nu \right) $ be a surjective cellular automaton
of radius $r$ with $\nu -$equicontinuous points. Suppose that $x$ is a $\nu $%
-equicontinuous point. Then, the set $\sigma ^{-p}\left( B_{\left[ -r,r%
\right] }\left( x\right) \right) \cap \left( B_{\left[ -r,r\right] }\left(
x\right) \right) $ is of positive measure for every integer $p.$ Moreover
for each point $y\in \sigma ^{-p}\left( B_{\left[ -r,r\right] }\left(
x\right) \right) \cap \left( B_{\left[ -r,r\right] }\left( x\right) \right) $
the sequence $F^{k}\left( y\right) \left( -r,r\right) $ is eventually
periodic.
\end{lemma}

\begin{proof}
Let $r$ be the radius of the cellular automaton and let $x$ be a $\nu $%
-equicontinuous point. Then we have $\nu \left( B_{\left[ -r,r\right]
}\left( x\right) \right) >0.$\newline

By the ergodicity of the shift, for a fixed value $p>0$, we have:%
\begin{equation*}
\nu \left( \sigma ^{-p}B_{\left[ -r,r\right] }\left( x\right) \cap \left( B_{%
\left[ -r,r\right] }\left( x\right) \right) \right) >0.
\end{equation*}%
For any $y\in \sigma ^{-p}B_{\left[ -r,r\right] }\left( x\right) \cap \left(
B_{\left[ -r,r\right] }\left( x\right) \right) $ we have:%
\begin{equation*}
\forall k\geq 0:F^{k}\left( y\right) \left( -r,r\right) =F^{k}\left(
y\right) \left( -r+p,r+p\right) .
\end{equation*}

Let us denote $y\left( -r,r\right) =w$ and by $u$ the word between two
occurrences of $w$. Let us also denote by $y^{-}$ the part at left of the
first word $w$ and $y^{+}$ that's at right of the second $w.$

By incorporating the word $uw$ in $y$ between the words $w$ and $u$ we
obtain a new element $y^{\left( 1\right) }$ containing two occurrences of
the word $uw.$By repeating this process we obtain a sequence $y^{\left(
i\right) }$ containing at each iteration one more occurrence of the word $uw.
$ 

By recurrence under the $F$-action it is possible to show that $y^{\left(
i\right) }$ still belongs to $B_{\left[ -ij_{1}-\left( i+1\right)
r,ij_{1}+\left( i+1\right) r\right] }\left( x\right) $ for any $i>0.$

The sequence of configurations $y^{\left( i\right) }=y^{-}w\underset{%
\leftarrow \text{ }\left( i-1\right) \text{ }\rightarrow }{uwuw...uw}y^{+}$
containing at each step a one more occurrence of the word $uw$ converge to
the shift periodic configuration $\left( uw\right) ^{\infty }$ which share
with $y$ the same coordinates over $\left( -r,r\right) $.

As the set $B_{\left[ -r,r\right] }\left( x\right) $ is closed the periodic
point $\left( uw\right) ^{\infty }$ is in $B_{\left[ -r,r\right] }\left(
x\right) ,$ the sequence of words $\left( F^{k}\left( x\right) \left(
-r,r\right) \right) _{k\geq 0}$ is then eventually periodic.
\end{proof}

\begin{proposition}
Let $\left( F,\nu \right) $ be a surjective cellular automaton with $\nu -$%
equicontinuous points then $F$ has a measurable equicontinuous factor which
is a cellular automaton.
\end{proposition}

\begin{proof}
Let $r$ be the radius of the cellular automaton and let $x$ be a $\nu $%
-equicontinuous point. Then we have $\nu \left( B_{\left[ -r,r\right]
}\left( x\right) \right) >0.$ Using Lemma \ref{Lemma} for every $y\in \sigma
^{-p}B_{\left[ -r,r\right] }\left( x\right) \cap \left( B_{\left[ -r,r\right]
}\left( x\right) \right) $ the sequence $F^{k}\left( y\right) \left(
-r,r\right) $ is eventually periodic.

As the number of words of length $2r+1$ is finite, so there exists a common
period $p$ and preperiod $p_{0}$ for all points of $\sigma ^{-p}B_{\left[
-r,r\right] }\left( x\right) \cap \left( B_{\left[ -r,r\right] }\left(
x\right) \right) .$

For some $y\in \sigma ^{-p}\left( B_{\left[ -r,r\right] }\left( x\right)
\right) \cap \left( B_{\left[ -r,r\right] }\left( x\right) \right) $
consider the finite set 
\begin{equation*}
\mathbb{P}=\left\{ F^{k}\left( y\right) \left( -r,r\right) :p_{0}\leq k\leq
p_{0}+p-1\right\} =\left\{ p_{k}:p_{0}\leq k\leq p_{0}+p-1\right\}
\end{equation*}%
Notice that the definition of the set $\mathbb{P}$ do not depend on the
choice of $y.$

Let us define the measurable sets 
\begin{equation*}
\widetilde{W}_{k}=F^{-1}\left\{ p_{k}\right\} :p_{0}\leq k\leq p_{0}+p-1.
\end{equation*}%
Consider the alphabet $A=\left( \mathbb{Z}/\left( p+1\right) \mathbb{Z}%
\right) $ and let the function $\pi $ be defined by 
\begin{equation*}
\forall x\in A^{\mathbb{Z}}:\left\{ 
\begin{array}{l}
\pi \left( x\right) _{i}=i-p_{0}:\left[ x\left( i,i+2r+1\right) \right]
\subset \widetilde{W}_{k}:k_{0}\leq k\leq k_{0}+p. \\ 
p:otherwise.%
\end{array}%
\right.
\end{equation*}%
The function $\pi $ is measurable and is associated to the equicontinuous
cellular automaton defined by :%
\begin{equation*}
\forall x\in \left( \mathbb{Z}/\left( p+1\right) \mathbb{Z}\right) ^{\mathbb{%
Z}}:C\left( x\right) _{i}=\left\{ 
\begin{array}{l}
\left( x_{i}+1\right) \func{mod}p\text{ }if\text{ }x\neq p \\ 
p\text{ }if\text{ }x=p%
\end{array}%
\right.
\end{equation*}%
Computing $\pi \circ F$ and $C\circ \pi $ we obtain :%
\begin{equation*}
\left( \pi \circ F\right) \left( x\right) =\left\{ 
\begin{array}{l}
\left( x_{i}+1\right) \,mod\,p:\left[ x\left( i,i+2r+1\right) \right]
\subset \widetilde{W}_{k}:k_{0}\leq k\leq k_{0}+p. \\ 
p:otherwise.%
\end{array}%
\right.
\end{equation*}%
\begin{equation*}
\left( C\circ \pi \right) \left( x\right) =\left\{ 
\begin{array}{l}
\left( x_{i}+1\right) \,mod\,p:\left[ x\left( i,i+2r+1\right) \right]
\subset \widetilde{W}_{k}:k_{0}\leq k\leq k_{0}+p. \\ 
p:otherwise.%
\end{array}%
\right.
\end{equation*}%
This shows that $C$ is a measurable equicontinuous factor of $F.$
\end{proof}

In this section we show that an ergodic CA cannot admit irrational
eigenvalues.

\begin{proposition}
Let $\left( A^{\mathbb{Z}},\mathbb{B},F,\nu \right) $ be a surjective
cellular automaton. If $F$ is ergodic then it cannot have any irrational
eigenvalue.
\end{proposition}

\begin{proof}
Let us suppose there is a subset $G_{0}$ of $A^{\mathbb{Z}}$ such that $\nu
\left( G_{0}\right) =1$ and $g$ is the eigenfunction associated to $e^{2i\pi
\alpha }$ with $\alpha \in \mathbb{R}\backslash \mathbb{Q}$ on $G_{0}.$
Since $\nu $ is $\sigma $-invariant the set $G=\tbigcap\limits_{n=-\infty
}^{\infty }\sigma ^{-n}\left( G_{0}\right) $ satisfies $\nu \left( G\right)
=1.$

We have then :%
\begin{equation*}
\forall x\in G:g\left( F\left( x\right) \right) =e^{2i\pi \alpha }g\left(
x\right) \Rightarrow g\left( F\left( \sigma ^{-n}\left( x\right) \right)
\right) =e^{2i\pi \alpha }g\left( \sigma ^{-n}\left( x\right) \right) .
\end{equation*}%
By commutation with the shift we obtain : 
\begin{equation*}
\dfrac{g\left( F\left( \sigma ^{-n}\left( x\right) \right) \right) }{g\left(
F\left( x\right) \right) }=\dfrac{g\left( \sigma ^{-n}\left( x\right)
\right) }{g\left( x\right) }\Rightarrow \dfrac{g\left( \sigma ^{-n}\left(
F\left( x\right) \right) \right) }{g\left( F\left( x\right) \right) }=\dfrac{%
g\left( \sigma ^{-n}\left( x\right) \right) }{g\left( x\right) }.
\end{equation*}

As $g$ is an eigenfunction the property $g\left( x\right) =0$ is invariant
by $F$ and the set $g^{-1}\left( \left\{ 0\right\} \right) $ is invariant.%
\newline
By ergodicity of $\nu $ the set $g^{-1}\left( \left\{ 0\right\} \right) $ is
either of measure 0 or of measure 1. As $g$ is non null on a set of positive
measure it is then almost everywhere non null. Thus the quotients are well
defined almost everywhere.

Let $h=\dfrac{g\circ \sigma ^{-n}}{g}$ we have then $h\left( F\left(
x\right) \right) =h\left( x\right) $ so $h$ is an invariant function for $F$
and by ergodicity it is constant.

In the following we will show that the function $h$ cannot be constant.

By the Lusin theorem for every $\eta >0$ there exists a closed set with $%
E_{\eta }\subset A^{\mathbb{Z}}$ and $\nu \left( A^{\mathbb{Z}}\backslash
E_{\eta }\right) \leq \eta $ and such that the restriction of $F$ to the set 
$E_{\eta }$ is continuous.

As $\nu $ is ergodic the eigenfunction $g$ is of constant module denoted by $%
r=\left\vert g\left( x\right) \right\vert .$

As $\alpha $ is irrational for every $\delta $ small enough there exist an
integer $p$ such that: $B_{\delta }\left( e^{2i\pi p\alpha }\right) \cap
B_{\delta }\left( e^{2i\pi \left( -p\right) \alpha }\right) =\emptyset .$
Where $B_{\delta }$ denote the ball of radius $\delta $ in $\mathbb{C}$.

Consider an arbitrary integer $q$ by continuity on $E_{\eta }$ there exists
two words $w_{1},w_{2}$ such that $\left\vert w_{1}\right\vert =\left\vert
w_{2}\right\vert $ satisfying :%
\begin{equation*}
\begin{cases}
\forall x\in \left[ w_{1}\right] \cap E_{\eta }:g\left( x\right) \in
B_{\delta }\left( re^{2i\pi \left( p+q\right) \alpha }\right) . \\ 
\forall x\in \left[ w_{2}\right] \cap E_{\eta }:g\left( x\right) \in
B_{\delta }\left( re^{2i\pi \left( q\right) \alpha }\right) .%
\end{cases}%
\end{equation*}%
Let be the words $w_{1}uw_{2}$ and $w_{1}uw_{1}$ with $u$ an arbitrary word
such that $\left\vert u\right\vert =n-\left\vert w_{1}\right\vert $.

We have then :%
\begin{equation*}
\begin{cases}
\forall x\in \left[ w_{1}uw_{2}\right] \cap E_{\eta }:g\left( x\right) \in
B_{\delta }\left( re^{2i\pi \left( p+q\right) \alpha }\right) \,. \\ 
\forall x\in \left[ w_{2}uw_{1}\right] \cap E_{\eta }:g\left( x\right) \in
B_{\delta }\left( re^{2i\pi q\alpha }\right) .%
\end{cases}%
\end{equation*}%
From an other side we have :%
\begin{equation*}
\begin{cases}
\forall x\in \left[ w_{1}uw_{2}\right] \cap E_{\eta }:g\left( \sigma
^{-n}\left( x\right) \right) \in B_{\delta }\left( re^{2i\pi q\alpha
}\right) . \\ 
\forall x\in \left[ w_{2}uw_{1}\right] \cap E_{\eta }:\,g\left( \sigma
^{-n}\left( x\right) \right) \in B_{\delta }\left( re^{2i\pi \left(
p+q\right) \alpha }\right) .%
\end{cases}%
\end{equation*}%
Consequently :%
\begin{equation*}
\begin{cases}
\forall x\in \left[ w_{1}uw_{2}\right] \cap E_{\eta }\Rightarrow h\left(
x\right) \in B_{\delta }\left( e^{2i\pi p\alpha }\right) . \\ 
\forall x\in \left[ w_{2}uw_{1}\right] \cap E_{\eta }\Rightarrow h\left(
x\right) \in B_{\delta }\left( e^{2i\pi \left( -p\right) \alpha }\right) .%
\end{cases}%
\end{equation*}%
Thus $h$ cannot be constant and $F$ cannot be ergodic.
\end{proof}

\section{Conclusion}

Two results are shown about equicontinuous factors of cellular automata. An
interesting direction is to try to characterize the maximal equicontinuous
factor of a cellular automaton. A first natural category to investigate
would be that of cellular automata with equicontinuity points. Is the
maximal equicontinuous factor still a cellular automaton or not is another
natural question to investigate.

We know that an ergodic cellular automaton is weakly mixing. The existence
of an example of a cellular automaton which is weakly mixing but not mixing
is still an open problem.

\end{document}